\documentclass[12pt]{article}
\usepackage{graphicx}
\usepackage{amsmath}
\usepackage{amsthm}
\usepackage{graphics}
\usepackage{epsfig}
\usepackage{amssymb}

\usepackage[sort&compress, authoryear]{natbib}

\usepackage[toc,page]{appendix}
\usepackage{xcolor}
 \usepackage{setspace}
\usepackage{extarrows}
\usepackage{subfig}
\usepackage{lmodern}

\usepackage{pdflscape}
\usepackage{subfig}
\usepackage{enumitem}
\allowdisplaybreaks

\newcommand{\bfb}{\mathbf{b}}
\newcommand{\bfP}{\mathbf{P}}

\newcommand{\R}{\mathbb{R}}
\newcommand{\lb}{\left(}
\newcommand{\rb}{\right)}
\newcommand{\Var}{\operatorname{Var}}
\newcommand{\cov}{\operatorname{Cov}}
\newcommand{\E}{\mathbb{E}}
\newcommand{\N}{\mathbb{N}}
\newcommand{\bfSigmahat}{	\hat{ \mathbf{\Sigma} }}
\newcommand{\bfIhat}{\hat{ \mathbf{I} }}
\newcommand{\bfR}{\mathbf{R}}
\newcommand{\bfQ}{\mathbf{Q}}
\newcommand{\bfSigma}{\mathbf{\Sigma}}
\newcommand{\bfX}{\mathbf{X}}
\newcommand{\bfI}{\mathbf{I}}
\newcommand{\bfY}{\mathbf{Y}}

\newcommand{\cond}{\stackrel{\mathcal{D}}{\to}}
\newcommand{\conp}{\stackrel{\mathbb{P}}{\to}}
\newcommand{\tr}{\operatorname{tr}}

\newcommand{\inv}{^{-1}}
\newcommand{\sq}{^{\frac{1}{2}}}

\newcommand{\PR}{\mathbb{P}}

\newcommand{\bfx}{\mathbf{x}}

\newcommand{\proj}{\operatorname{proj}}

\newtheorem{theorem}{Theorem}[section]
\newtheorem{lemma}{Lemma}[section]

\newtheorem{corollary}{Corollary}[section]

\newtheorem{remark}{Remark}[section]

\numberwithin{equation}{section}

\allowdisplaybreaks

\begin{document}
\title{Fluctuations of the diagonal entries of a large sample precision matrix}
\date{\today}
\author{Nina Dörnemann, Holger Dette}
\maketitle

\begin{abstract}
For a given $p\times n$ data matrix $\bfX_n$ with i.i.d. centered entries and a population covariance matrix $\bfSigma$, the corresponding sample precision matrix $\hat\bfSigma\inv$ is defined as the inverse of the sample covariance matrix $\hat\bfSigma = (1/n) \bfSigma^{1/2} \bfX_n \bfX_n^\top \bfSigma^{1/2}$.  We determine the joint distribution of a
vector of diagonal  entries of the matrix  $\hat\bfSigma\inv$ in the situation, where $p_n=p< n$, $p/n \to y \in [0,1)$ for $n\to\infty$ and $\bfSigma$ is a diagonal matrix. Remarkably, our results cover both the case where the dimension is negligible in comparison to the sample size and the case where it is of the same magnitude. 
  Our approach  is based on a QR-decomposition of the data matrix, yielding a connection to random quadratic forms and allowing the application of a central limit theorem. Moreover, we discuss an interesting connection to linear spectral statistics of the sample covariance matrix. More precisely, the logarithmic diagonal entry of the sample precision matrix can be interpreted as a difference of two highly dependent linear spectral statistics of $\hat\bfSigma$ and a submatrix of $\hat\bfSigma$. This difference of spectral statistics fluctuates on a much smaller scale than each single statistic.  
\end{abstract}

\section{Introduction} \label{sec_chap9}

Many statistical problems as they occur in biology or finance demand estimates of the covariance matrix or its inverse, for which the sample precision matrix is a popular choice. Spurred by the groundbreaking advances of data collecting devices, these  applications nowadays call for analysis tools of high-dimensional data sets \citep[see, e.g.,][and references therein]{Fan2006, Johnstone2006}. Moreover, they motivate the investigation of the probabilistic properties of large sample covariance or
precision matrices, where the dimension of the data  and  the sample size are of the same order.
In the last decades, the scientific interest was mainly focused on the probabilistic  properties of the spectrum of the  sample covariance matrix. Since the pioneering work of \cite{mp1967} on the empirical spectral distribution of $\hat\bfSigma$ for the case  $p/n \to y\in(0,\infty)$, the asymptotic behavior of its eigenvalues and eigenvalue statistics
has been studied by numerous authors. For example, we mention the works of \cite{baiyin1988} on the limiting spectral distribution in the case $y=0$, \cite{johnson1982}, \cite{baisilverstein2004}, \cite{ zheng_et_al_2015}, \cite{najimyao2016}
on linear spectral statistics, \cite{baik2006eigenvalues} on the eigenvalues of spiked population models,
and of \cite{Johnstone2001}, \cite{bai2008limit}  on the extreme eigenvalues of $\hat\bfSigma$, to name just a few. Some of these results can be used for the study of the spectrum or spectral statistics of $\hat\bfSigma\inv.$   
Furthermore, \cite{zheng2015clt} established a central limit theorem for linear spectral statistics of a rescaled version of the sample precision matrix. In the case
where the dimension exceeds the sample size, 
\cite {bodnar2016spectral}
investigated  the asymptotic properties of  linear spectral statistics of the Moore-Penrose inverse of the sample covariance matrix. 

From a statistical point of view, the sample precision matrix plays a vital role in the analysis of high-dimensional  linear models.
In particular, the diagonal elements of  the matrix  $\hat\bfSigma\inv$ are proportional to the conditional variances of the least squares estimator of the individual coefficients in the linear model (provided that the errors are independent and homoscedastic and there is no intercept in the model). 
Under the additional assumption of  a multivariate normal distribution, the exact distribution of $(\mathbf{\hat \Sigma}\inv)_{qr}$ is well-understood  for fixed dimension and sample size ($1 \leq q,r \leq p$). In fact, $n\inv \hat\bfSigma\inv $ follows an inverse Wishart distribution  \citep[see][for more details] {von1988moments,nydick2012wishart,gupta2018matrix}.
Apart from this, the asymptotic properties of 
$(\mathbf{\hat \Sigma}\inv)_{qq}$ for 
non-normal distributed data and a dimension growing with the sample sizes are not well understood so far.
We add to this line of research by establishing a central limit theorem for the diagonal entries of a large sample precision matrix.
Our approach is based on a consequence of Cramer's rule 
$
	(\hat\bfSigma\inv)_{qq} = | \hat\bfSigma^{(-q)} | / | \hat\bfSigma | , ~ 1 \leq q \leq p,
$
where $\mathbf{\hat \Sigma}^{(-q)}$ denotes the $(p-1) \times (p-1)$ submatrix of $\mathbf{\hat \Sigma}$ with the $q$th row and $q$th column being deleted. This representation reveals an explicit connection to a random quadratic form, which is shown to satisfy a central limit theorem.  
Moreover, we also observe an immediate connection to linear spectral statistics of sample covariance matrices: the  logarithm  of the $q$th diagonal entry $\log (\mathbf{\hat \Sigma}\inv)_{qq}$
is  a difference of  two linear spectral statistics of $\hat\bfSigma$ and its submatrix $\hat\bfSigma^{(-q)} \in \R^{(p - 1) \times (p - 1)}$.
However, due to the strong dependence between the eigenvalues of $\hat{\mathbf{\Sigma}} $ and $\hat{\mathbf{\Sigma}}^{(-q)}$, the asymptotic behavior of this statistic cannot be described by  the meanwhile classical CLT  of \cite{baisilverstein2004} or one of the many follow-up works.  Interestingly, the difference of spectral statistics fluctuates on a scale $1/\sqrt{n}$ which is of significantly smaller order than the fluctuations of each single linear spectral statistic $\log | \hat{\mathbf{\Sigma}} |$ and $\log | \hat{\mathbf{\Sigma}}^{(-q)} | $. 
More precisely, after appropriate normalization, a finite-dimensional vector of diagonal entries follows a multivariate normal distribution.
 Similarly  to linear spectral statistics of the sample covariance matrix, the limiting variance of $(\mathbf{\hat \Sigma}\inv)_{qq}$ is determined by the fourth moment of the underlying data generating distribution.

We conclude, mentioning that many authors have investigated the fluctuations of the entries of different types of random matrices or functions of random matrices. Exemplary, we mention  the work \cite{lytova2009fluctuations} on Gaussian random matrices, \cite{pizzo2012fluctuations, o2013fluctuations} on Wigner matrices and \cite{o2014fluctuations} on sample covariance matrices.  
A more closely related work to the results presented here is \cite{erdoes},
who considered  linear spectral statistics 
of the  sample covariance matrix and its minor from 
i.i.d. data with finite moments of any order. Choosing  the function $\log(x)$ in their main result and combining this with the delta method  gives a CLT for a single diagonal entry
of the sample precision matrix. 
In contrast to the work of these authors, our approach requires  only the existence of the fourth moment 
and also allows a proof of the weak 
convergence of a vector of diagonal entries
of the precision matrix.

The remaining part of this paper is organized as follows.
A CLT for a single diagonal entry is given in Section \ref{sec_vary} and is afterwards generalized to the joint convergence of several diagonal entries. All proofs of our main results are provided in Section \ref{sec_proof_vary} and Section \ref{sec_proof_joint_conv}. In Section \ref{sec_conclusion}, we give an outlook to future work concerning the sample precision matrix. Finally, 
Section \ref{sec_qr_decomp} in the Appendix sheds some
light on the QR-decomposition of the data matrix, which is an important tool used in the proofs.

	\section{A CLT for diagonal entries of the empirical precision matrix  } \label{sec_vary}
Throughout this paper, let 
\begin{align}
\mathbf{X}_n 
	=(x_{ij})_{\substack{i=1,...,p \\ j=1,...,n}}  
	\in \R^{p \times n}
	    \label{hol1}
\end{align}
denote  a random  $p  \times n$ matrix with i.i.d. centered entries having  a continuous distribution, $\mathbf{\Sigma} = \mathbf{\Sigma}_n \in \R^{p\times p}$ nonrandom and (symmetric) positive definite matrix with symmetric square root $\bfSigma^{1/2}$.  The matrix $\bfSigma$ denotes the population covariance matrix and for most of the following results, it is assumed to be a diagonal matrix (except for the normal case). We denote the sample covariance matrix by 
	\begin{align*}
	    \hat\bfSigma = \frac{1}{n} \bfSigma^{\frac{1}{2}} \bfX_n \bfX_n^\top \bfSigma^{\frac{1}{2}} \in \R^{p\times p}.
	\end{align*}
	If $p<n$, the inverse matrix $\hat\bfSigma\inv$ is almost surely well-defined and called the sample precision matrix. 
	We are now in the position to formulate the first  main result of this section. 

	\begin{theorem}[CLT for diagonal entries of full-sample precision matrix]\label{thm_prec}
	Let $\bfSigma \in\R^{p\times p}$ be a diagonal matrix with positive diagonal entries. 
	Assume that the random variables $\{ x_{ij} ~|~ 1 \leq i \leq p, ~ 1 \leq j \leq n \} $
	in \eqref{hol1}    
	are i.i.d. with continuous distribution, 
 $\E[x_{11}]=0, ~\Var (x_{11})=1$ and $\E[x_{11}^4] = \nu_4 < \infty$. Let $p/n \to y \in [0,1)$ for $n\to\infty$. 
	Then, it holds for $n\to \infty$ and $q\in\{1, \ldots, p\}$
		\begin{align*}
		\frac{ \sqrt{n - p + 1} }{ \lb \mathbf{\Sigma} \inv \rb_{qq} } \lb 
			 \frac{n - p + 1}{n} \lb \mathbf{\hat{\Sigma}} \inv \rb_{qq} - \lb \mathbf{\Sigma} \inv \rb_{qq}
			\rb 
		\cond \mathcal{N} ( 0, \rho ),~ n \to\infty,
	\end{align*}	
	where the asymptotic variance is given by $\rho = 2 + (\nu_4 - 3)	(1-y)$.
	\end{theorem}
	The proofs of  this  and of all 
   other results in this paper 
   are deferred to Section \ref{sec_proof_vary} and \ref{sec_proof_joint_conv}. 
   At this point, we only sketch the main
   arguments for  the proof of Theorem \ref{thm_prec}.
    We use  a QR-decomposition of the data matrix to derive a representation of the diagonal entry as the inverse of a quadratic form. With this knowledge at hand, we prove a CLT 
    for this quadratic form
    by an application of a central limit theorem for martingale difference schemes. By the delta method, we finally get asymptotic normality for $(\mathbf{\hat \Sigma}\inv)_{qq}$ being its inverse. 
    Note that  QR-decompositions appear in other contexts in random matrix theory. For example, \cite{wang2018} used this tool to derive the logarithmic law of the determinant of the sample covariance matrix for the case $p/n \to 1$, while \cite{heiny2021log} recently investigated the log-determinant of the sample correlation matrix under an infinite fourth moment. We also refer to  \cite{nguyen_vu} and \cite{bao2015},  who used the QR-decomposition to 
    	provide proofs of Girko's logarithmic law for a general random matrix with independent entries.

	\begin{remark}~~ \label{rem21}
	{\rm 
	\begin{enumerate}
	\item   Remarkably, our result  also covers the moderately high dimensional case $y=0$, where the dimension is negligible in comparison to the sample size. In this case, we may formulate the statement of Theorem \ref{thm_prec} as 
		\begin{align*}
		\frac{ \sqrt{n } }{ \lb \mathbf{\Sigma} \inv \rb_{qq} } \lb 
			 \lb \mathbf{\hat{\Sigma}} \inv \rb_{qq} - \lb \mathbf{\Sigma} \inv \rb_{qq}
			\rb 
		\cond W \sim \mathcal{N} ( 0, \nu_4 - 1 ),~ n \to\infty.
	\end{align*}	
	 \item 
	 As mentioned previously,
	 the statistic $\log (( \mathbf{\Sigma} \inv )_{qq} ) $ can be   interpreted as a difference of two linear spectral statistics of  sample covariance matrices and a CLT for this random variable would yield a CLT for $( \mathbf{\Sigma} \inv )_{qq}$ via the delta method.
		 Recently, \cite{erdoes} 
		 considered the  case  $\mathbf{\Sigma} = \mathbf{I}$
		 and 
		 developed a CLT for the difference of linear spectral statistics of a sample covariance matrix and its minor, which is applicable to a standardized and centered version of $\log ( \mathbf{\Sigma} \inv )_{qq}  $. 
		 Their result requires i.i.d. entries 
		 $x_{ij}$ with  finite moments of all order, while we only assume a finite fourth moment in Theorem \ref{thm_prec}. Moreover, 
		  in comparison to Theorem \ref{thm_prec}, their asymptotic regime  does not include the 
		  case $p/n \to 0$. 
		  Note that \cite{erdoes}
		 do not 
		 assume the existence of the limit $y$ of $p/n$. We only need this assumption to determine the limiting variance $\rho$, 
		 but it is not necessary for proving a CLT as in Theorem \ref{thm_prec}. One could instead normalize  by a factor  $1/\sqrt{\rho_n}$ defined in equation  \eqref{def_rho_n} in the proof of Theorem \ref{thm_prec}. 
		 We also  emphasize that the techniques used for proving Theorem \ref{thm_prec} 
		 sets us in the position to investigate the joint convergence of several diagonal elements of the sample precision matrix given in Theorem \ref{thm_joint} below.
		 \item   Furthermore, the entries of the empirical precision matrix can also be interpreted as entries of the resolvent matrix $\mathbf{D}(z) = ( \hat\bfSigma - z \mathbf{I})\inv$ for $z=0$. This draws an interesting connection to other existing works in this field. For $z \in \mathbb{C} \setminus \mathbb{R}$ and $y>0$, the fluctuations of the entries of $\mathbf{D}(z)$ are investigated in Theorem 5.1 by \cite{o2014fluctuations}. In this work, the asymptotic normality of the entries is concluded from a central limit theorem for quadratic forms (see Theorem 6.4 of \cite{benaych2011fluctuations}), while we concentrate on directly verifying the conditions of a central limit theorem for martingale difference schemes. We note that Theorem 6.4 of \cite{benaych2011fluctuations} is also applicable to our setting. However, for the sake of completeness, we will prove asymptotic normality via the martingale central limit theorem and thus, extend a result of \cite{bhansali2007} on central limit theorems for quadratic forms. 
	\end{enumerate}
	}
	\end{remark}

	The variance and mean structure of the limiting distribution of linear spectral statistics of sample covariance matrices are usually expressed via contour integrals and depend on the limiting spectral distribution of $\bfSigma$ in a subtle way \citep[see][]{baisilverstein2004, najimyao2016, panzhou2008}. So far, an explicit expression for these quantities has only been found in the null case $\bfSigma = \bfI$, and even for diagonal matrices 
		 as considered in Theorem \ref{thm_prec}, explicit expressions are out of reach. In this case, despite its close connection to these kinds of linear spectral statistics, the corresponding quantities of a diagonal entry $(\hat\bfSigma\inv)_{qq}$ 
		 depend asymptotically 
		 on its population version $(\bfSigma\inv)_{qq}$ in an explicit form. 
		 In particular,  for $\bfSigma = \operatorname{diag}(\bfSigma)$, the asymptotic mean and variance of a scaled diagonal entry $\sqrt{n-p} (\hat\bfSigma\inv)_{qq}/ (\bfSigma\inv)_{qq} $ do not depend on $\bfSigma\inv$ anymore. Moreover,
		 the following corollary, which is a direct 
		 consequence of Theorem \ref{thm_prec} and Lemma \ref{lem_cramer_norm} in Section \ref{sec_aux_thm},
		 shows that these statements are correct 
		  for  general  population covariance matrices
		  when imposing a normal assumption on the data.

	\begin{corollary} \label{cor_normal}
	Let $\bfSigma \in\R^{p\times p}$ be a symmetric positive definite matrix and assume that  the random variables $\{ x_{ij} ~|~ 1 \leq i \leq p, ~ 1 \leq j \leq n \} $
	in \eqref{hol1}    
	are i.i.d. with  	$x_{ij} \sim \mathcal{N}(0,1)$. 
	Then, it holds for $n\to \infty, p/n \to y \in [0,1)$ and $q\in\{1, \ldots, p\}$
		\begin{align*}
		\frac{ \sqrt{n - p + 1} }{ \lb \mathbf{\Sigma} \inv \rb_{qq} } \lb 
			 \frac{n - p + 1}{n} \lb \mathbf{\hat{\Sigma}} \inv \rb_{qq} - \lb \mathbf{\Sigma} \inv \rb_{qq}
			\rb 
		\cond \mathcal{N} ( 0, 2 ),~ n \to\infty.
	\end{align*}	
\end{corollary}		

Our final result of this section  provides the joint asymptotic distribution of two diagonal entries and is proven in Section \ref{sec_proof_joint_conv}.
\begin{theorem} \label{thm_joint}
Let $\bfSigma \in\R^{p\times p}$ be a diagonal matrix with positive diagonal entries. 
	Assume that  the random variables
    $\{ x_{ij} ~|~ 1 \leq i \leq p, ~ 1 \leq j \leq n \} $
	in \eqref{hol1}    
	are i.i.d. with continuous distribution,   $\E[x_{11}]=0, ~\Var (x_{11})=1$ and $\E[x_{11}^4] = \nu_4 < \infty$. Let $p/n \to y \in [0,1)$ for $n\to\infty$. 
	Then, it holds for $n\to \infty$ and $1 \leq q_1 \neq q_2 \leq p$
		\begin{align*}
		& \left\{ \frac{ \sqrt{n - p + 1} }{ \lb \mathbf{\Sigma} \inv \rb_{ii} } \lb 
			 \frac{n - p + 1}{n} \lb \mathbf{\hat{\Sigma}} \inv \rb_{ii} - \lb \mathbf{\Sigma} \inv \rb_{ii}
			\rb 
			\right\}_{i=q_1, q_2}^\top 
		\cond   \mathcal{N}_2 \lb \mathbf{ 0}, \rho \mathbf{I}_2 \rb,~ n \to\infty,
	\end{align*}	
	where $\rho = 2 + (\nu_4 - 3)	(1-y)$.
\end{theorem}

 \begin{remark}
 {\rm 
	Note that Theorem \ref{thm_joint} provides a nontrivial generalization of Theorem \ref{thm_prec} since the diagonal entries of the empirical precision matrix 
	are not independent. For more details on the concrete dependence structure, we refer the reader to Lemma \ref{lem_formula_inv} and  \ref{lem_rep_inv_quad_form} in Section \ref{sec_proof_joint_conv}. 
	Moreover, it is notable that these random variables are asymptotically independent.
In general, this property will not be valid beyond the diagonal case, and we can observe a proper dependency between two diagonal entries of the sample precision matrix. In particular, we know for the case of normally distributed data from the properties of the inverse Wishart distribution \citep[see, e.g.][]{von1988moments, press2005applied}  that 
$$
\cov \lb \sqrt{n-p} \frac{n-p}{n}  (\hat\bfSigma\inv)_{q_1,q_1}, \sqrt{n-p} \frac{n-p}{n} (\hat\bfSigma\inv)_{q_2,q_2} \rb = 2 (\bfSigma\inv)_{q_1,q_2} + o(1)  
$$
for $1 \leq q_1, q_2 \leq p$ and $p/n=\mathcal{O}(1)$.
}
 \end{remark}

	\section{Proof of Theorem \ref{thm_prec} } \label{sec_proof_vary}
	  In order to state the proofs rigorously, we need to introduce further notation. 
    We denote the columns of the random matrix $\bfX_n$ by $\bfx_1, \ldots, \bfx_n$ and the rows by $\bfb_1, \ldots, \bfb_p,$ that is, we write 
	\begin{align} \label{def_bp}
	\mathbf{X}_n 
	=(x_{ij})_{\substack{i=1,...,p \\ j=1,...,n}}  
	= (\mathbf{b}_1, \ldots, \mathbf{b}_p)^\top = (\mathbf{x}_1, \ldots, \mathbf{x}_n)
	\in \R^{p \times n}. 
\end{align} 
    In the case $\bfSigma = \bfI$, we denote the sample covariance matrix by
	\begin{align*}
		\hat{\mathbf{I}} = \frac{1}{n}  \mathbf{X}_n \mathbf{X}_n^\top
		= \frac{1}{n}\sum\limits_{i=1}^n \mathbf{x}_i \mathbf{x}_i^\top 
		\in \R^{p\times p}.
	\end{align*}
 In order to pursue the approach based on Cramer's rule as described in the introduction, we will introduce several submatrices. 
	If we set for some $q\in\{1, \ldots, p\}$
	\begin{align*}
		\tilde{\mathbf{X}}_n^{(-q)} = (\mathbf{b}_1, \ldots, \mathbf{b}_{q-1}, 
		\mathbf{b}_{q+1}, \ldots, \mathbf{b}_p)^\top
		\in \R^{(p-1) \times n},
	\end{align*}
	then 
	\begin{align*}
		\hat{ \mathbf{I}}^{(-q)} = \frac{1}{n} \tilde{\mathbf{X}}_n^{(-q)} \lb \tilde{\mathbf{X}}_n^{(-q)} \rb ^\top
		\in \R^{(p-1) \times (p-1)}
	\end{align*}
	can be obtained from $\hat{\mathbf{I}}$ by deleting the $q$th row and the $q$th column. 
	Similarly, if we set $\mathbf{Y}_n = \mathbf{\Sigma}\sq \mathbf{X}_n = (\mathbf{d}_1, \ldots, \mathbf{d}_p)^\top \in \R^{p \times n}$ and $\tilde{\mathbf{Y}}_n^{(-q)} = (\mathbf{d}_1, \ldots, \mathbf{d}_{q -1}, \mathbf{d}_{q+1}, \ldots, \mathbf{d}_p)^\top$, we define
	\begin{align*}
		\hat{\mathbf{\Sigma}} = \frac{1}{n} \mathbf{Y}_n \mathbf{Y}_n^\top
		\textnormal{ and }
		\hat{\mathbf{\Sigma}  }^{(-q)} = \frac{1}{n} \tilde{\mathbf{Y}}_n^{(-q)} \lb \tilde{\mathbf{Y}}_n^{(-q)} \rb^\top. 
	\end{align*}
	Additionally, the matrix $\mathbf{\Sigma}^{(-q)} \in \R^{ (p - 1) \times (p -1)}$ 
	can be obtained from $\mathbf{\Sigma}$ by deleting the $q$th row and the $q$th column. 
		
		We continue by proving Theorem \ref{thm_prec} using a CLT for martingale difference schemes. 
		The auxiliary results for these proofs can be found in Section \ref{sec_aux_thm}.

	\begin{proof}[Proof of Theorem \ref{thm_prec}]

Noting that $\bfSigma$ is a diagonal matrix and that the distribution of $\bfX_n$ is invariant under a permutation of the $q$th and the $p$th row, we see that
	\begin{align*}
		\frac{ \lb \hat{ \mathbf{\Sigma}} \inv \rb_{qq} } { \lb \mathbf{\Sigma} \inv \rb_{qq} }
		=  \lb \hat{ \mathbf{I}} \inv \rb_{qq} \stackrel{\mathcal{D}}{=} 
		\lb \hat{ \mathbf{I}} \inv \rb_{pp}
		= \frac{ \lb \hat{ \mathbf{\Sigma}} \inv \rb_{pp} } { \lb \mathbf{\Sigma} \inv \rb_{pp} }.
	\end{align*}
	Thus, we may assume $q=p$ without loss of generality. From now on, 	the proof is divided in several steps. 
	 \subsubsection*{Step 1: QR decomposition}

In this step, we rewrite $|\hat \bfI|$ and $| \hat{\bfI}^{(-p)} |$ in a more handy form via the QR decomposition. More details on this decomposition can be found in Section \ref{sec_qr_decomp}.

	 As explained in detail in Section \ref{sec_qr_decomp}, we get by proceeding the QR-decomposition for $\bfX_n^\top$ 
	 \begin{align} \label{qr_x}
		\mathbf{X}_n^\top = \mathbf{QR}, ~  \mathbf{X}_n = \mathbf{R}^\top \mathbf{Q}^\top,
	\end{align}	  
	where $\mathbf{Q}=(\mathbf{e}_1, \ldots, \mathbf{e}_p)\in\R^{n\times p}$ denotes a matrix with orthonormal columns satisfying $\mathbf{Q}^\top \mathbf{Q} = \mathbf{I}$ and $\mathbf{R}\in\R^{p\times p}$ is an upper triangular matrix with entries $r_{ij} = (\mathbf{e}_i, \mathbf{b}_j)$ for $i\leq j$ and $r_{ij} = 0$ for $ i > j$, $i,j\in\{1, \ldots, p\}$.
	Note that, since $\lb \tilde{\mathbf{X}}_n^{(-p)} \rb ^\top$ is the same as $\mathbf{X}_n^\top$ but with the $p$th column $\mathbf{b}_p$ removed, we have 
	  \begin{align} \label{qr_x_tilde}
		\lb \tilde{\mathbf{X}}_n^{(-p)} \rb^\top = \mathbf{Q} \tilde{\mathbf{R}}, ~ 
		 \tilde{\mathbf{X}}_n^{(-p)} = \tilde{\mathbf{R}}^\top \mathbf{Q}^\top,
	\end{align}	
	where $\tilde{\mathbf{R}} = (r_{ij})_{\substack{1 \leq i \leq p, \\1 \leq j \leq p-1}} \in \R^{ p \times (p - 1)} $ and we set  $\tilde{\mathbf{R}}^{(-p)} = (r_{ij})_{\substack{1 \leq i, j \leq p -1}} \in \R^{ (p -1) \times (p - 1)} $.
	Using \eqref{qr_x}, we write
	\begin{align*}
		| \mathbf{X}_n \mathbf{X}_n^\top | = | \mathbf{R}^\top \mathbf{Q}^\top \mathbf{Q} \mathbf{R} | = | \mathbf{R}^\top \mathbf{R} | 
		= | \mathbf{R} |^2
		= \prod\limits_{i=1}^p r_{ii}^2 
	\end{align*}
	and similarly, by using \eqref{qr_x_tilde} and the Cauchy-Binet formula,
	\begin{align*}
		\left| \tilde{\mathbf{X}}_n^{(-p)} \lb \tilde{\mathbf{X}}_n^{(-p)} \rb^\top \right| 
		= | \tilde{\mathbf{R}}^\top \tilde{\mathbf{R}} | 
		=  | \tilde{\mathbf{R}}^{(-p)} | ^2
		= \prod\limits_{\substack{i=1, \\ i \neq p}}^p r_{ii}^2.
	\end{align*}
	Thus, we obtain from 
	Cramer's rule and the fact that $\bfSigma$ is a diagonal matrix, 
	\begin{align}
		& \lb \frac{  \lb \hat \bfSigma \inv \rb_{pp} }{\lb  \bfSigma \inv \rb_{pp}} \rb\inv
		= \lb \lb  \hat\bfI\inv \rb_{pp} \rb\inv 
		= \frac{| \hat \bfI | }{| \hat\bfI^{(-p)} | }
		=  \frac{1}{n} r_{pp}^2. 
		\label{z2}
	\end{align}
	
	Before continuing with Step 2 of the proof of Theorem \ref{thm_prec}, we visit as an illustrating example the normal case where the distribution of $r_{pp}^2$ is explicitly known. 
	 \subsubsection*{Illustration: The normal case}
		If we assume additionally that $x_{ij} \sim \mathcal{N}(0,1)$ i.i.d. for $i\in\{1,\ldots,p\}, j\in\{1, \ldots, n\}$, then it is well-known that $r_{pp}^2 \sim \mathcal{X}_{n - p + 1}$ (see, e.g.,  \cite{goodman} or directly use \eqref{diag_r_quad_form}), that is, 
		\begin{align*}
			r_{pp}^2 \stackrel{\mathcal{D}}{=} \sum\limits_{j=1}^{n-p + 1} Z_j^2 ,
		\end{align*}
		where $Z_j$ are i.i.d. standard normal distributed random variables, $j\in\{1, \ldots, n - p + 1\}$. 
		Thus, we are able to apply a CLT for $r_{pp}^2$, namely,
		\begin{align*}
		& \sqrt{n-p + 1} \lb \frac{1}{n-p + 1} r_{pp}^2 - 1 \rb 
		= \frac{1}{\sqrt{n - p + 1}}\sum\limits_{j=1}^{n-p + 1} ( Z_j^2 - 1)
		\cond   \mathcal{N} (0, 2 ).
		\end{align*}
		Applying the delta method, we get
		\begin{align*}
			\sqrt{n - p + 1}    \lb \frac{n - p + 1}{ r_{pp}^2}  - 1 \rb
			\cond  \mathcal{N} ( 0 , 2).
		\end{align*}
		Thus, using \eqref{z2}, we conclude
		\begin{align}
			 & \frac{ \sqrt{n - p + 1} }{ \lb \mathbf{\Sigma} \inv \rb_{pp} } \lb 
			 \frac{n - p + 1}{n} \lb \mathbf{\hat{\Sigma}} \inv \rb_{pp} - \lb \mathbf{\Sigma} \inv \rb_{pp}
			\rb \nonumber \\
			 = &  \sqrt{n - p + 1} \lb \frac{n - p + 1}{n} \frac{\lb \mathbf{\hat{\Sigma}} \inv \rb_{pp}}{\lb \mathbf{\Sigma} \inv \rb_{pp}} - 1 \rb \nonumber \\
			 = & \sqrt{n - p + 1}    \lb \frac{n - p + 1}{ r_{pp}^2}  - 1 \rb
			  \cond \mathcal{N} ( 0 ,2). \label{normal_case}
		\end{align}
		Note that in the normal case, we have $\nu_4=3$.
		Thus, we have recovered the assertion of Theorem \ref{thm_prec} in this special case. 

	\subsubsection*{Step 2: CLT for quadratic forms} 
	In this step, we will show that the random variable $r_{pp}^2$ meets the conditions of a CLT for martingale difference schemes.
	In Section \ref{sec_qr_decomp}, it is shown that (see \eqref{diag_r_quad_form})
	\begin{align*}
		r_{pp}^2 = \mathbf{b}_p^\top \mathbf{P}(p - 1) \mathbf{b}_p,
	\end{align*}
	where $\mathbf{P}(0) = \mathbf{I}_n$ and for $q>1$
	\begin{align} \label{def_P}
		\mathbf{P}(q) = & \mathbf{I} - \tilde{\mathbf{X}}_{n,q}^\top \lb \tilde{\mathbf{X}}_{n,q}  \tilde{\mathbf{X}}_{n,q}^\top\rb\inv \tilde{\mathbf{X}}_{n,q} \in \R^{n\times n}
	\end{align}
	denotes the projection matrix onto the orthogonal complement of the subspace generated by the first $q$ rows of $\mathbf{X}_n$, that is, 
	that is,
 	 \begin{align*} 
 	 \tilde{\mathbf{X}}_{n,q} = & (\mathbf{b}_1, \ldots, \mathbf{b}_{q} )^\top \in \R^{q \times n} .
 	 \end{align*}
 	 Note that the random vector $\bfb_p$ is defined in \eqref{def_bp}. 
	For the following analysis, we denote $\mathbf{P}(p-1) = \mathbf{P} = (p_{ik})_{1 \leq i,k \leq n}$, which only depends on the random variables $\bfb_1, \ldots,  \bfb_{p-1}$ and is independent of $\bfb_p$.

		 We write
	\begin{align*}
		\sqrt{ \frac{ n - p + 1}{\rho_n} } \frac{1}{n - p + 1} \lb r_{pp}^2 - ( n - p  + 1) \rb 
		= & \frac{1}{\sqrt{ \rho_n (  n - p + 1 ) } } \lb \mathbf{b}_p^\top \bfP  	
		\mathbf{b_p}
 		- \E_{\mathbf{b}_q} \left[\mathbf{b}_q^\top \bfP
		\mathbf{b_p} \right] \rb \\
		= & 
		\frac{1}{\sqrt{\rho_n ( n - p + 1 ) } } \sum\limits_{i=1}^n Z_{pi},
	\end{align*}
	where for $i\in\{1, \ldots, n\}$, $n\in\N$
	\begin{align}
		Z_{pi} = & 
		 2 b_{pi} \sum\limits_{k=1}^{i-1} p_{ki} b_{pk} 
		+  p_{ii} \lb b_{pi}^2 - \E [ b_{pi}^2]  
		\rb 		, \nonumber \\
		\rho_n  = & 2 + \frac{\nu_4 - 3}{ n - p + 1} \sum\limits_{i=1}^n  p_{ii}^2. \label{def_rho_n}
	\end{align}
		  For $i\in\{1,\ldots,n\}$, let $\E_i$ denote the conditional expectation with respect to the $\sigma$-field $\mathcal{F}_{pi}$ generated by $\{ \mathbf{b}_1, \ldots, \mathbf{b}_{p-1}\} \cup \{ b_{pk} : 1 \leq k \leq i \}$. Furthermore, $\E_0 [ X ] = \E[X]$ denotes the usual expectation. \\
		  Since $b_{pk}$ is measurable with respect to $\mathcal{F}_{p,i-1}$ for $k\in\{1, \ldots, i-1 \}$ and $b_{pj}$  is independent of $\mathcal{F}_{p,i-1}$ for $j\in\{i, \ldots,n\}$, and $\mathbf{P}$ is measurable with respect to $\mathcal{F}_{pi}$ for all $i \in \{1, \ldots, n\}$, we obtain
		  \begin{align*}
		  	\E_{i - 1} [ Z_{pi}]
		  	= &  
		 2 \sum\limits_{k=1}^{i-1}  \E_{i - 1} [ b_{pi}  p_{ki} ] b_{pk} 
		+  \E_{i-1} \left[  p_{ii}  \lb b_{pi}^2 - \E [ b_{pi}^2] \rb \right]
	 \\
	   = & 
		 2 \E [ b_{pi} ] \sum\limits_{k=1}^{i-1}  p_{ki}  b_{pk} 
		+   p_{ii}  \lb \E [ b_{pi}^2 ] - \E [ b_{pi}^2] \rb  
		= 0, ~ 2 \leq i \leq n.
		  \end{align*}
		  Note that $Z_{pi}$ is measurable with respect to $\mathcal{F}_{pi}$ ($1 \leq i \leq n$). 
	These observations imply that for each $n\in\N$,  $(Z_{pi})_{1 \leq i \leq n}$ forms a martingale difference sequence with respect to the filtration $(\mathcal{F}_{pi})_{1 \leq i \leq n}$. This representation of a random quadratic form as a martingale difference scheme generalizes the one of \cite{bhansali2007}. Note that we are not able to apply their Theorem 2.1 directly in order to prove asymptotic normality, since in our case $\mathbf{P}$ is a random matrix and the random vectors $\mathbf{b}_p$ vary with $n\in\N$.
	Thus, we have to give a direct proof showing that it satisfies the conditions of the central limit theorem for martingale difference sequences provided in Lemma \ref{mds} in Section \ref{sec_aux_thm}.
	More precisely, we will show that for all $\delta >0$
	\begin{align}
	\sigma_n^2 = \frac{1}{\rho_n ( n - p + 1 ) } \sum\limits_{i=1}^n \E_{i - 1} [ Z_{pi}^2 ]
	\conp  1 , \label{cond1} \\
	r_n(\delta) = \frac{1}{ \rho_n ( n - p + 1 ) } \sum\limits_{i=1}^n \E 
	\left[ Z_{pi}^2 I_{\{ |Z_{pi}| \geq \delta \sqrt{ ( n - p + 1 ) \rho_n } \} }\right]
	\to 0,
	\label{cond2}
	\end{align}
	as $n \to \infty.$ \\
	As a preparation for the following steps, we note that 
	\begin{align}
		& \max\limits_{l=1, \ldots, n} \sum\limits_{m=1}^n p_{lm}^2
		\leq  || \mathbf{P} ||^2 \leq 1, \label{ineq_1}\\
		& \tr \lb \mathbf{P}^2 \rb 
		=  \sum\limits_{i,k = 1}^n p_{ki} p_{ik}
		= || \mathbf{P} ||_2^2 = \tr \mathbf{P} = n - p + 1,
		\label{ineq_2}
	\end{align}
	where $||\mathbf{P}||$ denotes the spectral norm of $\mathbf{P}$
	and $||\mathbf{P}||_2$ denotes the Frobenius norm of $\mathbf{P}$. The first inequality in \eqref{ineq_1} is a well-known estimate for general symmetric matrices and can be shown by choosing the unit vectors for the maximum appearing in the definition of the spectral norm, while the equality in \eqref{ineq_2} follows from the fact that $\mathbf{P}^2 = \mathbf{P}$.  
	\subsubsection*{Step 2.1: Calculation of the variance} 
	We begin with a proof of \eqref{cond1}.
	For this purpose, we calculate
	\begin{align}
		\sigma_n ^2 = & \frac{1}{ \rho_n ( n - p + 1 )} \sum\limits_{i=1}^n \E_{i - 1} \left[ 
		Z_{pi}^2
		\right] \nonumber \\
		= & \frac{4}{ \rho_n ( n - p + 1) } \sum\limits_{i=1}^n \E_{i - 1} \left[   \lb  \sum\limits_{k=1}^{i-1} p_{ki} b_{pk} \rb^2 \right] \nonumber \\
		& + \frac{4}{\rho_n ( n - p + 1)} \sum\limits_{i=1}^n \left\{ \lb \E \left[ b_{pi}^3 \right] - \E [ b_{pi} ] \E \left[ b_{pi}^2 \right] \rb 
		 \sum\limits_{k=1}^{i-1} b_{pk} \E_{i-1}[p_{ki}   p_{ii} ] \right\} \ \nonumber \\
		& +\frac{1}{\rho_n ( n - p + 1)} \sum\limits_{i=1}^n \E_{i-1} [ p_{ii}^2]  \E\big [b_{pi}^2 - \E[b_{pi}^2] \big]^2  
\nonumber \\
 = & \frac{4}{ \rho_n ( n - p + 1) } \sum\limits_{i=1}^n \E_{i - 1} \left[   \lb  \sum\limits_{k=1}^{i-1} p_{ki} b_{pk} \rb^2 \right] 
		 + \frac{4\E \left[ b_{p1}^3 \right] }{\rho_n ( n - p + 1)} \sum\limits_{i=1}^n \left\{  
		 \sum\limits_{k=1}^{i-1} b_{pk} p_{ki}   p_{ii}  \right\} \ \nonumber \\
		& +\frac{(\nu_4 - 1)}{\rho_n ( n - p + 1)} \sum\limits_{i=1}^n p_{ii}^2 
		. \label{sigma_n_sq}
	\end{align}
	  Here, we used  that $b_{pk}$ is measurable with respect to $\mathcal{F}_{pi}$ for $k\in\{1, \ldots, i \}$ and $b_{pj}$  is independent of $\mathcal{F}_{pi}$ for $j\in\{i+1, \ldots,n\}$, and $\mathbf{P}$ is measurable with respect to $\mathcal{F}_{pi}$ for all $i \in \{1, \ldots, n\}$.
	Moreover, we obtain using \eqref{ineq_2} 
	\begin{align}	
		1
		= &  \rho_n\inv \lb  2 + \frac{\nu_4 - 3}{ n - p + 1} \sum\limits_{i=1}^n  p_{ii}^2 \rb  \nonumber \\ 
	= & \frac{2}{\rho_n ( n - p + 1)} \sum\limits_{\substack{i,k=1, \\ i \neq k}}^n         p_{ki}^2 
		+\frac{\nu_4 - 1}{\rho_n ( n - p + 1)} \sum\limits_{i=1}^n    p_{ii}^2   \nonumber \\
		 = & \frac{4}{\rho_n ( n - p + 1)} \sum\limits_{i=1}^n     \sum\limits_{k=1}^{i-1}     p_{ki}^2 
		+\frac{\nu_4 - 1}{\rho_n ( n - p + 1)} \sum\limits_{i=1}^n  p_{ii}^2 . \label{mean_sigma_n_sq}   
	\end{align}

	Denoting $\nu_4 = 1 + \varepsilon$ for some small $\varepsilon>0$, we note that $\rho_n$ is uniformly bounded away from $0$, since for all $n\in\N$
	\begin{align}
		\rho_n 
		= 2 - \frac{2 - \varepsilon}{n - p + 1} \sum\limits_{i=1}^n  p_{ii}^2 
		\geq 2 - \frac{2 - \varepsilon}{n - p + 1} \sum\limits_{i=1}^n p_{ii}
		= \varepsilon >0. \label{bound_rho}
	\end{align}	
	
In the following, we will show that \eqref{cond1} holds true. 
	For this purpose, we write using \eqref{sigma_n_sq}, \eqref{mean_sigma_n_sq} and \eqref{bound_rho}
	\begin{align}
		| \sigma_n^2 - 1 | 
		\leq & 
		\frac{4}{ \rho_n ( n - p + 1) } \left| \sum\limits_{i=1}^n \lb   \E_{i-1} \left[     \sum\limits_{k=1}^{i-1} p_{ki} b_{pk} \right]^2 - \sum\limits_{k=1}^{i - 1}     p_{ki}^2  \rb \right| \nonumber \\
		& + \frac{4  \E | b_{p1}|^3  }{\rho_n ( n - p + 1)}  \left| \sum\limits_{i=1}^n \left\{  \sum\limits_{k=1}^{i-1} b_{pk} p_{ki}   p_{ii}  \right\} \right| 
		\nonumber \\ 
		\lesssim & \frac{1}{n - p + 1} \lb \delta_{n,1} + \delta_{n,2} + \delta_{n,3}\rb  ,
		\label{decomp_sigma}
	\end{align}
	where
	\begin{align*}
	\delta_{n,1} = & \left| \sum\limits_{i=1}^n 
	\sum\limits_{1 \leq k < j \leq i - 1} 
	  p_{ki} p_{ji}  b_{pk} b_{pj} \right|, \\
	\delta_{n,2} = & \left| \sum\limits_{i=1}^n  \sum\limits_{k=1}^{i-1}
	\lb b_{pk}^2 - 1  \rb  p_{ki}^2 \right|, \\ 
	\delta_{n,3} = & \left| \sum\limits_{i=1}^n  \sum\limits_{k=1}^{i-1} b_{pk}  p_{ki}   p_{ii}  \right|.  \\
	\end{align*}
	Similarly as in \cite{bhansali2007}, one can show that $ \delta_{n,i} / (n - p + 1) = o_{\PR} (1)$, as $n\to\infty$ for $i\in\{1,2,3\}$, by bounding the second moments of $\delta_{n,1}, \delta_{n,2}, \delta_{n,3}$.  
	Exemplarily, we demonstrate this for the term $\delta_{n,3}$. 
		Notice that an application of Lemma 2.1 in \cite{bhansali2007} and \eqref{ineq_2} yields
		\begin{align} \label{ineq_tilde_p}
		\lb \sum\limits_{i,i'=1}^n \lb \sum_{k=1}^{\min(i,i')-1} p_{ik} p_{i'k}  \rb^2 \rb^{\frac{1}{2}}
		 \lesssim \sqrt{n - p + 1} || \bfP || \leq \sqrt{n - p + 1.}
		\end{align}
		Using the Cauchy-Schwarz inequality, \eqref{ineq_tilde_p} and \eqref{ineq_2}, 
	\begin{align*}
		\E [ \delta_{n,3}^2 ] 
		& =  \E \left[ \sum_{i,i'=1}^n p_{ii} p_{i'i'} \sum\limits_{k=1}^{\min(i,i')-1} p_{ki} p_{ki'} \right]  \\
		& \leq \E \left[ \lb \sum_{i}^n p_{ii}^2 \rb \lb \sum_{i,i'=1}^n \lb \sum\limits_{k=1}^{\min(i,i')-1} p_{ki} p_{ki'} \rb^2 \rb^{\frac{1}{2}}  \right] \\
		& \lesssim \lb n - p + 1 \rb^{\frac{3}{2}} = o \lb ( n - p + 1)^2 \rb, ~ n\to\infty. 
	\end{align*}
	Proceeding similarly for the remaining terms $\delta_{n,1}$ and $\delta_{n,2}$, we get $\sigma_n^2 = 1 + o_{\PR} (1)$ as $n\to\infty$. 	By an application of Lemma \ref{lem_con_pii} given at the end of this section, the normalizing term $\rho_n$ converges in probability towards $\rho$ as $n\to\infty$. 
	
 \subsubsection*{Step 2.2: Verifying the Lindeberg-type condition \eqref{cond2}} 
	 Using a truncation argument as in \cite{bhansali2007}, it is sufficient to prove \eqref{cond2} under the assumption $ \E [ b_{11}^8 ] < \infty$. 
	Then, we obtain by using \eqref{bound_rho}
	\begin{align*}
	r_n ( \delta) \leq & 
	\frac{1}{(n - p + 1)^2 \rho_n^2 \delta^2} \sum\limits_{i=1}^n \E 
	\left[ Z_{pi}^4\right] 
	\lesssim    J_1 + J_2 ,
	\end{align*}
	where
	\begin{align*}
	J_1 = & \frac{1}{(n - p + 1)^2 \delta^2} \sum\limits_{i=1}^n  \E \left[
	 b_{pi}^4 \lb \sum\limits_{k=1}^{i-1} p_{ki} b_{pk} \rb^4 \right]
	\\ \lesssim & \frac{1}{(n - p + 1)^2 \delta^2} \sum\limits_{i=1}^n  \E \left[
	  \lb \sum\limits_{j,k=1}^{i-1} p_{ki} p_{ji} b_{pk} b_{pj} \rb^2 \right]
	\\ \lesssim &   \frac{1}{(n - p + 1) \delta^2} \sum\limits_{i=1}^n  \E \left[
	  \lb \sum\limits_{k=1}^{i-1} p_{ki}^2 b_{pk}^2 \rb^2 \right] \\ & 
	  +   \frac{1}{(n - p + 1) \delta^2} \sum\limits_{i=1}^n  \E \left[
	  \lb \sum\limits_{\substack{j,k=1 \\ j <k }}^{i-1} p_{ki} p_{ji} b_{pk} b_{pj} \rb^2 \right]
	,\\
	J_2 =& \frac{1}{(n - p + 1)^2 \delta^2}  \sum\limits_{i=1}^n\E \left[ p_{ii}^4 \lb b_{pi}^2 - \E [ b_{pi}^2] \rb^4 \right] 
	\lesssim \frac{1}{(n - p + 1) \delta^2} \sum\limits_{i=1}^n \E \left[ p_{ii}^4  \right].
	\end{align*}
	This implies using \eqref{ineq_1} and \eqref{ineq_2}
	\begin{align*}
	J_1 + J_2 \lesssim & 
	 \frac{1}{(n - p + 1)^2 \delta^2} \sum\limits_{i=1}^n
	\lb \sum\limits_{j,k=1}^{i -1} \E [ p_{ki}^2 p_{ji}^2 ]  + \E [ p_{ii}^4 ]
	\rb 
	\lesssim \frac{1}{(n - p + 1)^2 \delta^2} \sum\limits_{i,j,k=1}^n \E [ p_{ki}^2 p_{ji}^2 ] \\
	\lesssim  &   \frac{1}{(n - p + 1)^2 \delta^2} \sum\limits_{j,k=1}^n \E \left[ p_{jk}^2 \max\limits_{l=1, \ldots, n} \sum\limits_{m=1}^n p_{lm}^2 \right] \\
	\lesssim &   \frac{1}{(n - p + 1)^2 \delta^2}  \sum\limits_{j,k=1}^n \E [p_{jk}^2]
	= o(1) .
	\end{align*}
	\subsubsection*{Step 3: Conclusion via delta method} 
	In Step 2, we have shown that an appropriately centered and standardized version of $r_{pp}^2$ satisfies a CLT. By applying the delta method and using \eqref{z2}, 
	 we conclude that
	\begin{align*}
	  & \frac{ \sqrt{n - p + 1} }{ \lb \mathbf{\Sigma} \inv \rb_{pp} } \lb 
			 \frac{n - p + 1}{n} \lb \mathbf{\hat{\Sigma}} \inv \rb_{pp} - \lb \mathbf{\Sigma} \inv \rb_{pp}
			\rb 
			 =  \sqrt{n - p + 1}    \lb \frac{n - p + 1}{ r_{pp}^2}  - 1 \rb 
			 \\ & \cond \mathcal{N}(0,\rho) , ~ n\to\infty,
	\end{align*}
	which finishes the proof of Theorem \ref{thm_prec}. 
	\end{proof}

	\subsection{Auxiliary results} \label{sec_aux_thm}
		If $\bfSigma$ is a diagonal matrix, then it holds for $1 \leq q \leq p$
		\begin{align*}
			\lb \hat \bfI \inv \rb_{qq} = \frac{  \lb \hat \bfSigma \inv \rb_{qq} }{\lb  \bfSigma \inv \rb_{qq}}. 		\end{align*}
	This connection can be generalized to the case of dependent coordinates if we assume that the data follows a standard normal distribution. 
	\begin{lemma} \label{lem_cramer_norm}
	If $\bfSigma$ is a general (not necessarily diagonal) $p\times p$ population covariance matrix and $x_{ij} \stackrel{\textnormal{i.i.d.}}{\sim} \mathcal{N}(0,1)$ ($1\leq i \leq p, ~ 1 \leq j \leq n$), then for any $1 \leq q \leq p$
	\begin{align*}
		\frac{  \lb \hat \bfSigma \inv \rb_{qq} }{\lb  \bfSigma \inv \rb_{qq}}
			\stackrel{\mathcal{D}}{=} \lb \hat \bfI \inv \rb_{qq} . 
	\end{align*}
	\begin{proof}[Proof of Lemma \ref{lem_cramer_norm}]
		Let $( \bfSigma^{1/2} ) ^{(-q, \cdot)}$ denote the $(p-1)\times p$ submatrix of $\bfSigma^{1/2}$ where the $q$th row is deleted. Similarly, $( \bfSigma^{1/2} ) ^{(\cdot, -q)}$ denotes the $p\times (p-1)$ submatrix of $\bfSigma^{1/2}$ where the $q$th column is deleted. Using these definitions, we see that 
		\begin{align} \label{c1}
			\hat{\mathbf{\Sigma}}^{(-q)}
			= ( \bfSigma^{1/2} ) ^{(-q, \cdot)} \bfX_n \bfX_n^\top 
			( \bfSigma^{1/2} ) ^{(\cdot, -q)}
			= ( \bfSigma^{1/2} ) ^{(-q, \cdot)} \bfX_n \lb ( \bfSigma^{1/2} ) ^{(-q, \cdot)}  \bfX_n \rb ^\top .
		\end{align}
	Combining \eqref{c1} with the normal assumption, we have that 
		\begin{align*}
		( \bfSigma^{1/2} ) ^{(-q, \cdot)} \bfx_i  \sim \mathcal{N} ( \mathbf{0}, \bfSigma^{(-q)}),~ 1 \leq i \leq n,
		\end{align*}
		where we used that
		\begin{align*}
			( \bfSigma^{1/2} ) ^{(-q, \cdot)}  \lb ( \bfSigma^{1/2} ) ^{(-q, \cdot)}  \rb^\top 
			= ( \bfSigma^{1/2} ) ^{(-q, \cdot)}  ( \bfSigma^{1/2} ) ^{( \cdot, -q )} 
			= \bfSigma^{(-q)}.
		\end{align*}
		This implies that 
		\begin{align*}
			 ( \bfSigma^{(-q)} )^{1/2}  \tilde\bfX_n^{(-q)}  \stackrel{\mathcal{D}}{=}  \tilde{\bfY}_n^{(-q)} .
		\end{align*}
		Using Cramers rule, we get
		\begin{align*} 
			\frac{  \lb \hat \bfSigma \inv \rb_{qq} }{\lb  \bfSigma \inv \rb_{qq}}
			= \frac {  | \mathbf{\Sigma} | }{ | \hat{\mathbf{\Sigma}} | }
		\frac{ | \hat{\mathbf{\Sigma}}^{(-q)} | } {  |\mathbf{\Sigma}^{(-q)} | }
		\stackrel{\mathcal{D}}{=} \frac{| \hat \bfI^{(-q)} | }{ | \hat\bfI | } 
		= \lb \hat \bfI \inv \rb_{qq}. 
		\end{align*}
		The proof of Lemma \ref{lem_cramer_norm} concludes. 
	\end{proof}
	
	\end{lemma} 	
	
	In order to prove asymptotic normality of the quadratic forms appearing in the previous proofs, we make use of the following CLT for martingale difference schemes.
	
		\begin{lemma}[Theorem 35.12 in \cite{billingsley1995}] \label{mds}
		Suppose that for each $n\in\N$, $Z_{n1}, ..., Z_{nr_n}$ form a real martingale difference sequence with respect to the increasing $\sigma$-field $(F_{nj})$ having second moments. If, as $n\to \infty$
		\begin{equation}\label{mds_ass1}
			\sum\limits_{j=1}^{r_n} \E [ Z_{nj}^2 | F_{n,j-1} ] \stackrel{\mathbb{P}}{\to} \sigma^2,
		\end{equation}
		where $\sigma^2>0$, and for each $\varepsilon >0$,
		\begin{equation} \label{lindeberg}
			\sum\limits_{j=1}^{r_n} \E [ Z_{nj}^2 I_{\{|Z_{nj}| >\varepsilon\}} ] \to 0,
		\end{equation}
		then
		\begin{align*}
			\sum\limits_{j=1}^{r_n} Z_{nj} \cond \mathcal{N}(0,\sigma^2).
		\end{align*}
		\end{lemma}

	We conclude this section by proving the following lemma, which was used in the proof of Theorem \ref{thm_prec} and provides the limiting variance. 
	\begin{lemma} \label{lem_con_pii}
	It holds
	\begin{align*}
		\rho_n \conp \rho, ~n\to\infty,
	\end{align*}
	where $\rho$ is defined in Theorem \ref{thm_prec} and $\rho_n$ in \eqref{def_rho_n}.
	\end{lemma}
	
	\begin{proof} [Proof of Lemma \ref{lem_con_pii}]
	Assume that $y = 0.$ For this case, we note that
	\begin{align} 
		\frac{1}{n} \sum\limits_{i=1}^n p_{ii}^2
		& = \frac{1}{n} \sum\limits_{i=1}^n \lb 1 - p_{ii}\rb^2
		- 1 + \frac{2}{n} \sum\limits_{i=1}^n p_{ii}
		=  \frac{2 ( n - p + 1 ) }{n} - 1 + o_{\mathbb{P}}(1) \nonumber \\
		& = 1 + o_{\mathbb{P}}(1) 
		,~n\to\infty, \label{z3}
	\end{align}
	where we used
	\begin{align*}
		\frac{1}{n} \sum\limits_{i=1}^n \E (1 - p_{ii})^2 \leq 
		\frac{1}{n} \sum\limits_{i=1}^n \E [1 - p_{ii}] = \frac{1}{n} \tr ( \mathbf{I} - \mathbf{P} ) = \frac{p- 1}{n} = o(1), ~n\to\infty. 
	\end{align*}
	Then, \eqref{z3} implies
	\begin{align*}
		\rho_n = 2 + \frac{(\nu_4 - 3) n}{n - p + 1} + o_{\mathbb{P}} (1) = \nu_4 - 1 =  \rho. 
	\end{align*}
	Let $y \in ( 0,1)$. 
	Then we have from Theorem 3.2 in \cite{anatolyev_yaskov_2017}
	\begin{align*}
		\frac{1}{n} \sum\limits_{i=1}^n (1 - p_{ii} - y)^2 \conp 0, ~ n\to\infty,
\end{align*}	 
	which implies
	\begin{align*}
		\frac{1}{n} \sum\limits_{i=1}^n p_{ii}^2 
		& =  \frac{1}{n} \sum_{i=1}^n ( 1 - p_{ii} - y)^2 
		- (1-y)^2 + \frac{2 ( 1 -y )}{n} \sum\limits_{i=1}^n p_{ii} \\
		& = \frac{2 ( 1 -y ) ( n - p + 1) }{n} - (1 - y)^2 + o_{\mathbb{P}}(1) 
		 = ( 1 - y)^2 + o_{\mathbb{P}}(1), ~ n\to\infty. 
	\end{align*}
	We conclude for $n\to\infty$
	\begin{align*}
		\rho_n = 2 + \frac{(\nu_4 - 3) (1-y)^2 n}{n - p + 1} + o_{\mathbb{P}}(1) 
		= 2 + (\nu_4 -3) ( 1 -y) + o_{\mathbb{P}}(1) = \rho + o_{\mathbb{P}}(1).
	\end{align*}
	\end{proof}

\section{Proof of Theorem \ref{thm_joint}} \label{sec_proof_joint_conv}

\begin{proof}[Proof of Theorem \ref{thm_joint}] 
Since the distribution of $\hat\bfI\inv$ is invariant under interchanging rows of $\mathbf{X}_n$, we have 
	\begin{align*}
	\lb \frac{ \lb \hat{ \mathbf{\Sigma}} \inv \rb_{q_1,q_1} } { \lb \mathbf{\Sigma} \inv \rb_{q_1,q_1} } , 
		\frac{ \lb \hat{ \mathbf{\Sigma}} \inv \rb_{q_2,q_2} } { \lb \mathbf{\Sigma} \inv \rb_{q_2,q_2} }
	\rb
		& =  \lb  \lb \hat{ \mathbf{I}} \inv \rb_{q_1,q_1} ,
		\lb \hat{ \mathbf{I}} \inv \rb_{q_2,q_2} 		\rb 
		\stackrel{\mathcal{D}}{=} 
		\lb  \lb \hat{ \mathbf{I}} \inv \rb_{pp} ,
		\lb \hat{ \mathbf{I}} \inv \rb_{p-1,p-1} 		\rb  .
		\end{align*}
Thus, we may assume $q_1=p$ and $q_2=p - 1$ without loss of generality. We define
\begin{align} \label{def_Q}
		\bfQ(p) = \frac{\bfP(p-2) \bfb_p \bfb_p^\top \bfP(p-2) }{\bfb_p^\top \bfP(p-2) \bfb_p}.
	\end{align}
	 Similar to the proof of Theorem \ref{thm_prec}, we start by investigating the asymptotic properties of 
	 	 \begin{align*}
	 	 W_n  = & \Bigg\{ \frac{1}{\sqrt{n - p + 1} } \lb n \lb \frac {\lb  \bfSigmahat \inv \rb_{pp}}{\lb  \bfSigma \inv \rb_{pp}}\rb \inv - (n - p + 1) \rb , \\
	 	 &
	 	 \frac{1}{\sqrt{n - p + 1} } \lb n \lb \frac {\lb  \bfSigmahat \inv \rb_{p-1,p-1}}{\lb  \bfSigma \inv \rb_{p-1,p-1}}\rb \inv - (n - p + 1) \rb
			\Bigg\}^\top \\ 
	 	 = & \Bigg\{ \frac{1}{\sqrt{n - p + 1} } \lb n \lb  \bfIhat\inv \rb_{pp}\inv - (n - p + 1) \rb , \\ & 
	 	 \frac{1}{\sqrt{n - p + 1} } \lb n \lb  \bfIhat\inv \rb_{p-1,p-1}\inv - (n - p + 1) \rb
			\Bigg\}^\top \\
			= &  \frac{1}{\sqrt{n - p + 1 }} \Bigg\{ \bfb_p^\top \bfP(p-1) \bfb_p -  ( n - p + 1) , \\ & \bfb_{p-1}^\top ( \bfP(p-2) - \bfQ(p) ) \bfb_{p-1} - (n - p + 1) \Bigg\}^\top,
	 	 \end{align*}
	 	 where we used Lemma \ref{lem_rep_inv_quad_form}.  For the following analysis, we will use the fact that it is a projection matrix of rank one and independent of $\bfb_{p-1}$.
	 From now on, the proof is divided in several steps. 
\subsubsection*{Approximation and martingale difference scheme}
	Note that for any rank-one projection matrix $\bfQ\in \R^{n \times n}$ independent of $\bfb_p$, we have
		\begin{align*}
			\Var ( \bfb_p^\top \bfQ \bfb_p ) 
			 \lesssim 1 ~ \forall n\in\N,
\end{align*}		 
	and consequently, by Slutsky's lemma, it is sufficient to investigate
	\begin{align*}
		 W_{n}^{(1)} & = \frac{1}{\sqrt{n - p + 1 }} \left\{ \bfb_p^\top \bfP(p-2) \bfb_p -  ( n - p + 2) , \bfb_{p-1}^\top \bfP(p-2) \bfb_{p-1} - (n - p + 2 )  \right\}^\top \\  & = W_n + o_{\mathbb{P}} (1) .
	\end{align*}
	Throughout the rest of this proof, we denote $\bfP(p-2) = \bfP = (p_{ij})_{1 \leq i,j \leq n}$. 
	By an application of the Cramer-Wold device, we note that it is sufficient to prove a one-dimensional central limit theorem for
	\begin{align*}
		W_n^{(2)} = \frac{1}{\sqrt{n - p + 1 }} \left\{ a \lb  \bfb_p^\top \bfP \bfb_p -  ( n - p + 2) \rb + b \lb \bfb_{p-1}^\top \bfP\bfb_{p-1} - (n - p + 2 )  \rb \right\}^\top, ~a,b\in\R,
	\end{align*}
	in order to ensure that the vector $W_n^{(1)}$ converges to a two-dimensional normal distribution. We write
	\begin{align*}
		\frac{1}{\sqrt{\rho_n}} W_n^{(2)} = & \frac{1}{\sqrt{ (n - p + 1 ) \rho_n }}\sum\limits_{i=1}^n W_{pi} ,
	\end{align*}
	where
	\begin{align*}
		W_{pi} = & 
		 a\lb 2 b_{pi} \sum\limits_{k=1}^{i-1} p_{ki} b_{pk} 
		+  p_{ii} \lb b_{pi}^2 - 1
		\rb 	\rb 	
		+b \lb 2 b_{p-1,i} \sum\limits_{k=1}^{i-1} p_{ki} b_{p-1,k} 
		+  p_{ii} \lb b_{p-1,i}^2 - 1 
		\rb 	\rb	 , \\
		\rho_n  = & 2 + \frac{\nu_4 - 3}{ n - p + 1} \sum\limits_{i=1}^n  p_{ii}^2. 
	\end{align*} 
	For $p\in \N$, $1\leq i\leq n$, let $\mathcal{A}_{pi}$ denote the $\sigma$ field generated by  $\{ \mathbf{b}_1, \ldots, \mathbf{b}_{p-2}\} \cup \{ b_{pk}, b_{p-1,k} : 1 \leq k \leq i \}$.
	Similar to in the proof of Theorem \ref{thm_prec}, one can show that $(W_{pi})_{1 \leq i \leq n}$ forms a martingale difference sequence with respect to the $\sigma$-fields $(\mathcal{A}_{pi})_{1\leq i \leq n}$ for each $p\in\N$.
	In order to apply the central limit theorem given in Lemma \ref{mds}, we need to verify the conditions \eqref{mds_ass1} and \eqref{lindeberg}. 

\subsubsection*{Calculation of the variance}
We begin with a proof of condition \eqref{mds_ass1}. For simplicity, we write $\mathcal{A}_{pi}=\mathcal{A}_i$ for $1 \leq i \leq n$ and $p\in\N.$
Note that
\begin{align*}
	& \frac{1}{\rho_n (n - p + 1)} \sum\limits_{i=1}^n \E [ W_{pi}^2 | \mathcal{A}_{i-1} ] \\
	= &  \frac{a^2}{\rho_n (n - p + 1)} \E \left[ \lb 2 b_{pi} \sum\limits_{k=1}^{i-1} p_{ki} b_{pk} 
		+  p_{ii} \lb b_{pi}^2 - 1
		\rb\rb^2 \Big| \mathcal{A}_{i-1}\right] \\
		& +  \frac{b^2}{\rho_n (n - p + 1)} \E \left[ \lb 2 b_{p-1,i} \sum\limits_{k=1}^{i-1} p_{ki} b_{p-1,k} 
		+  p_{ii} \lb b_{p-1,i}^2 - 1 
		\rb  \rb^2 \Big| \mathcal{A}_{i-1} \right] \\
		& + \frac{2ab}{\rho_n (n - p + 1)}  \E \Bigg[ \lb 2 b_{p-1,i} \sum\limits_{k=1}^{i-1} p_{ki} b_{p-1,k} 
		+  p_{ii} \lb b_{p-1,i}^2 - 1 
		\rb  \rb \\ 
		& \times \lb 2 b_{pi} \sum\limits_{k=1}^{i-1} p_{ki} b_{pk} 
		+  p_{ii} \lb b_{pi}^2 - 1
		\rb\rb \Big| \mathcal{A}_{i-1} \Bigg] \\
		= &  \frac{a^2}{\rho_n (n - p + 1)} \E \left[ \lb 2 b_{pi} \sum\limits_{k=1}^{i-1} p_{ki} b_{pk} 
		+  p_{ii} \lb b_{pi}^2 - 1
		\rb\rb^2 \Big| \mathcal{A}_{i-1}\right] \\
		& +  \frac{b^2}{\rho_n (n - p + 1)} \E \left[ \lb 2 b_{p-1,i} \sum\limits_{k=1}^{i-1} p_{ki} b_{p-1,k} 
		+  p_{ii} \lb b_{p-1,i}^2 - 1 
		\rb  \rb^2 \Big| \mathcal{A}_{i-1}\right] \\
		= & a^2 + b^2 + o_{\mathbb{P}}(1) , ~ n\to\infty,
\end{align*}
where we used \eqref{cond1} from the proof of Theorem \ref{thm_prec}. 

\subsubsection*{Verification of the Lindeberg-type condition}

For a proof of condition \eqref{lindeberg}, we use the results from Step 2.2 in the proof of Theorem \ref{thm_prec} and obtain
\begin{align*}
	& \frac{1}{ \rho_n ( n - p + 1 ) } \sum\limits_{i=1}^n \E 
	\left[ W_{pi}^2 I_{\{ |W_{pi}| \geq \delta \sqrt{ ( n - p + 1 ) \rho_n } \} }\right] 
	\leq  \frac{1}{(n - p + 1)^2 \rho_n^2 \delta^2} \sum\limits_{i=1}^n \E 
	\left[ W_{pi}^4\right] \\
	\lesssim & \frac{a^4}{(n - p + 1)^2 \rho_n^2 \delta^2} \sum\limits_{i=1}^n \E 
	\left[ \lb 2 b_{pi} \sum\limits_{k=1}^{i-1} p_{ki} b_{pk} 
		+  p_{ii} \lb b_{pi}^2 - 1
		\rb\rb^4 \right] \\
		+&  \frac{b^4}{(n - p + 1)^2 \rho_n^2 \delta^2} \sum\limits_{i=1}^n \E 
	\left[ \lb 2 b_{p-1,i} \sum\limits_{k=1}^{i-1} p_{ki} b_{p-1,k} 
		+  p_{ii} \lb b_{p-1,i}^2 - 1 
		\rb  \rb^4 \right] 
		=  o(1), ~ n\to\infty.
\end{align*}
\end{proof} 

\subsubsection*{Conclusion via delta method}
Summarizing the steps above, we obtain from Lemma \ref{mds}
\begin{align*}
	W_n \cond \mathcal{N}_2 ( \mathbf{0}, \rho \mathbf{I}_2 ) , n\to\infty. 
\end{align*}
	By an application of the multivariate delta method, we have 
	\begin{align*}
			( Z_{n,p} , Z_{n,p-1} )^\top 
			= & \Bigg\{  \sqrt{n - p + 1} \lb 
			 \frac{n - p + 1}{ \bfb_p^\top \bfP(p-1) \bfb_p } - 1
			\rb ,  \\
			 & \sqrt{n - p + 1} \lb 
			 \frac{n - p + 1}{\bfb_{p-1}^\top ( \bfP(p-2) - \bfQ(p) ) \bfb_{p-1}} -1
			\rb 
			\Bigg\} ^\top 
			\\ \cond  & \mathcal{N}_2 ( \mathbf{0}, \rho \mathbf{I}_2 ) , n\to\infty. 
	\end{align*}

	\subsection{Auxiliary results}

The following lemma gives a concrete representation for any diagonal element of the sample precision matrix in terms of the entries of the triangular matrix $\mathbf{R}$.
\begin{lemma} \label{lem_formula_inv}
For $1 \leq q \leq p$, it holds 
	\begin{align*}
		n \lb \bfIhat\inv \rb_{qq}\inv 
		= r_{qq}^2 \prod\limits_{i=q+1}^p \frac{r_{ii}^2}{r_{ii,q}^2},
	\end{align*}
	where the matrix $\bfR$ is defined in the proof of Theorem \ref{thm_prec} and, 
	\begin{align*}
		r_{ii,q}^2 = \bfb_i^\top \bfP (i-1,q) \bfb_i, ~1 \leq i \neq q \leq p.
	\end{align*}
	Here, $\bfP(i-1,q)$ denotes the projection matrix on the orthogonal complement of span($\{\bfb_1, \ldots, \bfb_{i-1}\} \setminus \{ \bfb_q \} $ ). 
	In particular, if $q=p-1$, we obtain 
	\begin{align*}
		n \lb \bfIhat\inv \rb_{p-1,p-1}\inv 
		= \frac{r_{pp}^2 r_{p-1,p-1}^2}{\bfb_p^\top \bfP(p-2) \bfb_p}.
	\end{align*}

\end{lemma} 
\begin{proof} [Proof of Lemma \ref{lem_formula_inv}]
	Recall the QR-decomposition of $\bfX_n^\top$ given in Section \ref{sec_qr_decomp} and the resulting formula
	\begin{align*}
		| \bfX_n \bfX_n^\top | = \prod\limits_{i=1}^p r_{ii}^2.
	\end{align*}
	Note that the first $(q-1)$ step in the QR-decomposition of the matrices $\tilde{\bfX}_n^\top=(\tilde{\bfX}_n^{(-q)})^\top$ and $\bfX_n^\top$ coincide, which implies
	\begin{align*}
		| \tilde{\bfX}_n \tilde{\bfX}_n^\top | 
		= \prod_{i=1}^{q-1} r_{ii}^2 \prod_{i=q+1}^p r_{ii,q}^2.
	\end{align*}
 Combining these formulas with Cramer's rule, we conclude
	\begin{align*}
		n \lb \bfIhat\inv \rb_{qq}\inv  =  \frac{| \bfX_n \bfX_n^\top |}{| \tilde{\bfX}_n \tilde{\bfX}_n^\top |} = r_{qq}^2 \prod\limits_{i=q+1}^p \frac{r_{ii}^2}{r_{ii,q}^2}.
	\end{align*}
	
\end{proof}

Recall from the proof of Theorem \ref{thm_prec} (or see Section \ref{sec_qr_decomp} for more details)  that
\begin{align*}
	\lb \bfIhat\inv \rb_{pp}\inv = \frac{1}{n} r_{pp}^2 = \frac{1}{n} \bfb_p^\top \bfP(p-1) \bfb_p,
\end{align*}
while it follows from the fact the entries $x_{ij}$ of the matrix $\bfX_n$ are i.i.d. random variables that 
\begin{align*}
\lb \bfIhat\inv \rb_{qq}\inv  \stackrel{\mathcal{D}}{=} \lb \bfIhat\inv \rb_{pp}\inv, ~ 1\leq q \leq p.
\end{align*} 
Thus, these quantities can also be written as quadratic forms.
The next lemma provides such a representation and specifies the dependency structure between two diagonal elements. 
 Moreover, it helps us to understand the dependence structure between two diagonal entries and, thus, is crucial for proving Theorem \ref{thm_joint}.
For convenience, we restrict ourselves to the case $q=p-1$. 
\begin{lemma} \label{lem_rep_inv_quad_form}
It holds
\begin{align*}
		n \lb \bfIhat\inv \rb_{p-1,p-1}\inv = \bfb_{p-1}^\top \lb \bfP(p-2) - \bfQ(p) \rb \bfb_{p-1} ,
\end{align*}
	where $\bfP(p-2) - \bfQ(p)$ is a projection matrix of rank $n-p+1$ and independent of $\bfb_{p-1}$. For a precise definition of $\mathbf{Q}(p)$, we refer the reader to \eqref{def_Q}. 
	
\end{lemma} 
\begin{proof} [Proof of Lemma \ref{lem_rep_inv_quad_form}]
	Recall from Lemma \ref{lem_formula_inv} that
		\begin{align*}
		n \lb \bfIhat\inv \rb_{p-1,p-1}\inv 
		= \frac{r_{pp}^2 r_{p-1,p-1}^2}{\bfb_p^\top \bfP(p-2) \bfb_p}.
	\end{align*}
	Note that $\bfP(p-1) \bfb_p = \bfP(p-2) \bfb_p - \operatorname{proj}_{e_{p-1}} (\bfb_{p} ),$ where the projection of a vector $\mathbf{a}\in\R^n$ onto a vector $\mathbf{e} \in\R^n$ is given by
	\begin{align*}
		\proj_{\mathbf{e}} (\mathbf{a}) = \frac{(\mathbf{e}, \mathbf{a})}{(\mathbf{e}, \mathbf{e})} \mathbf{e}
\end{align*}	  
and (for details, see Section \ref{sec_qr_decomp})
\begin{align*}
		\mathbf{u}_{p-1} =\bfP(p-2) \bfb_{p-1}, ~ \mathbf{e}_{p-1} = \frac{\mathbf{u}_{p-1}}{|| \mathbf{u}_{p-1}||_2}.  
	\end{align*}
	Thus, we obtain
	\begin{align*}
		n \lb \bfIhat\inv \rb_{p-1,p-1}\inv 
		= & \bfb_{p-1}^\top \bfP(p-2) \bfb_{p-1} \lb 1 - \frac{ \bfb_p^\top \proj_{\mathbf{e}_{p-1}} (\bfb_p)}{\bfb_p ^\top \bfP(p-2) \bfb_p } \rb \\
		= & \bfb_{p-1}^\top \bfP(p-2) \bfb_{p-1}  \lb 1 - \bfb_p^\top \frac{(\mathbf{u}_{p-1}, \bfb_p )}{ (\mathbf{u}_{p-1}, \mathbf{u}_{p-1}) \bfb_p ^\top \bfP(p-2) \bfb_p  } \mathbf{u}_{p-1} \rb \\
		= & \bfb_{p-1}^\top \bfP(p-2) \bfb_{p-1}  - \frac{ \bfb_p^\top (\mathbf{u}_{p-1}, \bfb_p )\mathbf{u}_{p-1} }{\bfb_p ^\top \bfP(p-2) \bfb_p } \\
		= & \bfb_{p-1}^\top \bfP(p-2) \bfb_{p-1} - \bfb_{p-1}^\top \mathbf{Q}(p) \bfb_{p-1}.
	\end{align*}
	Note that $\mathbf{Q}(p)^2 = \mathbf{Q}(p)$ and $\bfP(p-2)\mathbf{Q}(p) = \mathbf{Q}(p) \bfP(p-2)= \mathbf{Q}(p)$. Consequently, we obtain
	\begin{align*}
		\lb \bfP(p-2) - \mathbf{Q}(p) \rb ^2
		= & \bfP(p-2)^2 + \mathbf{Q}(p)^2 - \bfP(p-2)\mathbf{Q}(p)  - \mathbf{Q}(p)  \bfP(p-2) \\
		= & \bfP(p-2) + \mathbf{Q}(p) - 2 \mathbf{Q}(p) = \bfP(p-2) - \mathbf{Q}(p).
 	\end{align*}
	This implies that $\bfP(p-2) - \mathbf{Q}(p)$ is a projection matrix independent of $\bfb_{p-1}$ of rank
	\begin{align*}
		\tr \lb \bfP(p-2) - \mathbf{Q}(p)\rb = n - p + 2 - 1 = n - p +1. 
	\end{align*}
\end{proof}

\section{Conclusions} \label{sec_conclusion}

In this paper, we have  provided a multivariate central limit theorem for the diagonal entries of a sample precision matrix if the dimension-to-sample-size ratio satisfies $p/n \to y \in[0,1)$ as $n\to\infty$ and the population covariance matrix is diagonal.
An important direction of future research is 
to find the asymptotic distribution of the  diagonal entries 
for a general structure of the population covariance matrix. 
We emphasize that this question results in a substantially more complicated problem, since 
the method of the proofs used  in this work 
is tailored to the diagonal case. In particular, we reduce the diagonal case $\bfSigma = \operatorname{diag}(\bfSigma)$ to the null case $\bfSigma= \bfI$. For a general distribution and a general population covariance matrix, this step is no longer correct. Then again, if we conduct a QR-decomposition for $(\bfSigma^{1/2} \bfX_n)^\top$ instead of $\bfX_n^\top$
(as in  step 1 of the proof of Theorem \ref{thm_prec}), we obtain a quadratic form where the random vectors depend on the projection matrix in an implicit form. Our proofs, especially the martingale argument for applying a CLT, rely crucially on the fact that the random vector $\bfb_p$ (defined in \eqref{def_bp}) is independent of the random projection matrix $\bfP(p-1)$ (defined in \eqref{def_P}). 
Similarly, the techniques used in \cite{erdoes}, 
which can be used to derive a central limit theorem for a single diagonal entry of the sample precision matrix by a representation as  a difference of two linear spectral statistics
(see Remark \ref{rem21}), require the even stronger assumption $\bfSigma = \bfI$.   
Additionally, it is not straightforward to adapt the tools provided by \cite{baisilverstein2004} due to the different normalizations appearing in the CLT for a single linear spectral statistic and the difference of two. 
The development of novel techniques that meet the challenges of the dependent case $\bfSigma \neq \textnormal{diag} (\bfSigma )$  will be the objective of our future work.

\bigskip

\textbf{Acknowledgements.} 
This work  was partially supported by the  
 DFG Research unit 5381 {\it Mathematical Statistics in the Information Age}, project number 460867398.  The authors would like to thank Giorgio Cipolloni and  L\'{a}szl\'{o} Erd\H{o}s for some helpful  discussions.

	\setlength{\bibsep}{1pt}
\begin{small}
\bibliography{references}
\end{small}

\appendix
\section{Details on the QR-decomposition of $\bfX_n^\top$}	\label{sec_qr_decomp}

In this section, we give more details on the QR-decomposition of the matrix $\bfX_n^\top$ \citep[compare Section 2 in][]{wang2018} and provide an explicit representation of the diagonal elements of $\bfR$ as a quadratic form in the rows of $\bfX_n$. \\ 
	To begin with, we describe the QR-decomposition of a general full-column rank matrix $\mathbf{A} = ( \mathbf{a}_1, \ldots, \mathbf{a}_p)\in \R^{n \times p}$ by applying the Gram-Schmidt procedure to the vectors $\mathbf{a}_1 , \ldots, \mathbf{a}_p$. 
	 Recall the definition of
	 the projection of a vector $\mathbf{a}\in\R^n$ onto a vector $\mathbf{e} \in\R^n, \mathbf{e} \neq \mathbf{0},$ is given by
	\begin{align*}
		\proj_{\mathbf{e}} (\mathbf{a}) = \frac{(\mathbf{e}, \mathbf{a})}{(\mathbf{e}, \mathbf{e})} \mathbf{e}.
\end{align*}	  
	It holds
\begin{align*}
	 {\displaystyle {\begin{aligned}\mathbf {u} _{1}&=\mathbf {a} _{1},&\mathbf {e} _{1}&={\frac {\mathbf {u} _{1}}{\|\mathbf {u} _{1}\|}},\\\mathbf {u} _{2}&=\mathbf {a} _{2}-\operatorname {proj} _{\mathbf {u} _{1}}\mathbf {a} _{2},&\mathbf {e} _{2}&={\frac {\mathbf {u} _{2}}{\|\mathbf {u} _{2}\|}},\\\mathbf {u} _{3}&=\mathbf {a} _{3}-\operatorname {proj} _{\mathbf {u} _{1}}\mathbf {a} _{3}-\operatorname {proj} _{\mathbf {u} _{2}}\mathbf {a} _{3},&\mathbf {e} _{3}&={\frac {\mathbf {u} _{3}}{\|\mathbf {u} _{3}\|}},\\&\;\;\vdots &&\;\;\vdots \\\mathbf {u} _{n}&=\mathbf {a} _{n}-\sum _{j=1}^{n-1}\operatorname {proj} _{\mathbf {u} _{j}}\mathbf {a} _{n},&\mathbf {e} _{n}&={\frac {\mathbf {u} _{n}}{\|\mathbf {u} _{n}\|}}.\end{aligned}}} 
	 \end{align*}
	 Rearranging these equations, we may write $ \mathbf{A} = \mathbf{QR}$,
	 where $\mathbf{Q}=(\mathbf{e}_1, \ldots, \mathbf{e}_p)\in\R^{n\times p}$ denotes a matrix with orthonormal columns satisfying $\mathbf{Q}^\top \mathbf{Q} = \mathbf{I}$ and $\mathbf{R}\in\R^{p\times p}$ is an upper triangular matrix with entries $r_{ij} = (\mathbf{e}_i, \mathbf{a}_j)$ for $i\leq j$ and $r_{ij} = 0$ for $ i > j$, $i,j\in\{1, \ldots, p\}$.

	 In order to ensure formal correctness of the QR decomposition for the matrix $\mathbf{X}_n^\top = (\mathbf{b}_1, \ldots, \mathbf{b}_p) $, we note that the matrix $\mathbf{X}_n^\top $ has full column rank since we assumed that each $x_{ij}$ follows a continuous distribution for $1\leq i \leq p, ~ 1 \leq j \leq n$. 
	   Performing the QR decomposition for the special choice $ \mathbf{A} = \mathbf{X}_n^\top = (\mathbf{b}_1, \ldots, \mathbf{b}_p) $, we get  
	 \begin{align*} 
		\mathbf{X}_n^\top = \mathbf{QR},
	\end{align*}	  
	where $\mathbf{Q}=(\mathbf{e}_1, \ldots, \mathbf{e}_p)\in\R^{n\times p}$ denotes a matrix with orthonormal columns satisfying $\mathbf{Q}^\top \mathbf{Q} = \mathbf{I}$ and $\mathbf{R}\in\R^{p\times p}$ is an upper triangular matrix with entries $r_{ij} = (\mathbf{e}_i, \mathbf{b}_j)$ for $i\leq j$ and $r_{ij} = 0$ for $ i > j$, $i,j\in\{1, \ldots, p\}$.
	Using the definitions $r_{qq}^2=(\mathbf{e}_i, \bfb_i)^2$ for $1 \leq q \leq p$ and $\bfP(0) = \mathbf{I}$, we have 
	\begin{align*}
		r_{11}^2  = (\mathbf{e}_1, \bfb_1)^2 = || \bfb_1||_2^2 = \bfb_1^\top \bfP(0) \bfb_1,
	\end{align*}
	and for $2 \leq q \leq p$
	\begin{align} \label{diag_r_quad_form}
		r_{qq}^2 
		= ( \mathbf{e}_{q}, \mathbf{b}_q)^2 
		= \lb \frac{ \mathbf{u}_q^\top \bfb_q}{|| \mathbf{u}_q ||_2 } \rb^2 
		= \lb \frac{\bfb_q^\top \bfP(q-1) \bfb_q }{|| \bfP(q-1) \bfb_q||_2 } \rb^2 
		= \mathbf{b}_q^\top \mathbf{P}(q - 1) \mathbf{b}_q, 
	\end{align}
	where the projection matrix $\bfP(q-1)$ is defined in \eqref{def_P} and satisfies $\bfP(q-1)^2 = \bfP(q-1)$.

	\end{document}